\documentclass[reqno, 11pt]{amsart}

\usepackage[text={155mm,235mm},centering]{geometry}
\input diagxy
\input xy

\usepackage[active]{srcltx} 

\newtheorem{thm}{Theorem}[section]
\newtheorem{cor}[thm]{Corollary}
\newtheorem{lem}[thm]{Lemma}
\newtheorem{prop}[thm]{Proposition}
\newtheorem{quest}[thm]{Question}

\theoremstyle{definition}

\theoremstyle{property}

\theoremstyle{remark}
\newtheorem{rem}[thm]{Remark}

\numberwithin{equation}{section}
\usepackage[mathscr]{eucal}
\usepackage[all]{xy}
\usepackage{mathrsfs}
\usepackage{xypic}
\usepackage{amsfonts}
\usepackage{amsmath}
\usepackage{amsthm,bm}
\usepackage{amssymb,tikz-cd}
\usepackage{latexsym}
\usepackage{tabularx}
\usepackage{graphicx}
\usepackage{pict2e}
\usepackage[pagebackref]{hyperref}
\usepackage{tikz}
\usepackage{textcomp}
\usepackage[scr=rsfs]{mathalfa}
\definecolor{ceruleanblue}{rgb}{0.16, 0.32, 0.75}
\hypersetup{colorlinks=true,allcolors=ceruleanblue}
\allowdisplaybreaks

\begin{document}

\title[The heredity and bimeromorphic invariance]{The heredity and bimeromorphic invariance of the  $\partial\bar{\partial}$-lemma property}

\author{Lingxu Meng}
\address{Department of Mathematics, North University of China, Taiyuan, Shanxi 030051,  P. R. China}
\email{menglingxu@nuc.edu.cn}%
\date{\today}



\begin{abstract}
  We give a simple proof of a result on the $\partial\bar{\partial}$-lemma property under a blow-up transformation by Deligne--Griffiths--Morgan--Sullivan's criterion.
  Here, we use an explicit blow-up formula for Dolbeault cohomology given in our previous work, which can be induced by a morphism expressed on the level of spaces of forms and currents.
  At last, we discuss the heredity and bimeromorphic invariance of the  $\partial\bar{\partial}$-lemma property.
\end{abstract}

\maketitle


%
\section{Introduction}
In non-K\"ahler geometry, the heredity and bimeromorphic invariance of the $\partial\bar{\partial}$-lemma property are two interesting problems, extensively studied in \cite{A, ASTT, DGMS, F, RYY, St1, St2, YY} especially in the recent days. The \emph{$\partial\bar{\partial}$-lemma} on a compact complex manifold $X$ refers
to that for every pure-type $d$-closed form on $X$, the properties of
$d$-exactness, $\partial$-exactness, $\bar{\partial}$-exactness and
$\partial\bar{\partial}$-exactness are equivalent while a compact complex manifold is called a \emph{$\partial\bar{\partial}$-manifold} if the \emph{$\partial\bar{\partial}$-lemma} holds on it.
\begin{quest}[Heredity]\label{q-1}
Does any closed complex  submanifold of an $n$-dimensional $\partial\bar{\partial}$-manifold still satisfy the $\partial\bar{\partial}$-lemma?
\end{quest}
\begin{quest}[Bimeromorphic invariance]\label{q-2}
Does any compact complex manifold being bimeromorphic to an $n$-dimensional $\partial\bar{\partial}$-manifold  satisfy the  $\partial\bar{\partial}$-lemma?
\end{quest}

Clearly, the heredity is true for the $\partial\bar{\partial}$-manifolds of dimensions $\leq 2$.
Suppose that $\widetilde{X}$ is a modification of a compact complex manifold $X$.
A. Parshin \cite{P} and P. Deligne, Ph. Griffiths, J. Morgan, D. Sullivan \cite{DGMS} proved that if $\widetilde{X}$ is a $\partial\bar{\partial}$-manifold, then so is $X$.
L. Alessandrini \cite{A} posed a question in its inverse direction: if $X$  satisfies the $\partial\bar{\partial}$-lemma, so does  $\widetilde{X}$?
We can easily  prove that, Question 1.2 is equivalent to Alessandrini's one.
It is true on complex surfaces by the classical results that each compact complex surface with even first Betti number is K\"ahler (see \cite{buch,lam} for a uniform proof) and the first Betti number is a bimeromorphic invariant, while the case of threefolds was first proved by S. Rao, S. Yang,  X.-D. Yang \cite{RYY} using a  Dolbeault blow-up formula  and S. Yang,  X.-D. Yang \cite{YY} using a Bott-Chern blow-up formula. The general case is still open. For any nonnegative integer $k\leq n$, we weaken Question \ref{q-1} as
\begin{quest}[Heredity for codimension $\geq k$]\label{q-3}
Does any closed  complex submanifold of codimension $\geq k$ of an $n$-dimensional $\partial\bar{\partial}$-manifolds still satisfy the $\partial\bar{\partial}$-lemma?
\end{quest}
For convenience, Questions \ref{q-1}-\ref{q-3} are denoted by $(\textrm{$H_n$})$, (\textrm{$B_n$}) and $(\textrm{$H_{n,k}$})$, respectively. Obviously, $(\textrm{$H_n$})=(\textrm{$H_{n,0}$})\Leftrightarrow(\textrm{$H_{n,1}$})$ and if $k_1\leq k_2$, then $(\textrm{$H_{n,k_1}$})\Rightarrow(\textrm{$H_{n,k_2}$})$.

P. Deligne \emph{et al.} \cite[(5.21)]{DGMS} gave an important  result, which related the $\partial\bar{\partial}$-lemma property  with Hodge filtration and the degeneracy of the Fr\"{o}licher spectral sequence at $E_1$-page.
S. Rao, S. Yang and X.-D. Yang \cite[Theorem 1.6]{RYY} investigated the bimeromorphic invariance of the degeneracy of Fr\"{o}licher spectral sequence at $E_1$ by their Dolbeault blow-up formula and pointed out that these results are applicable to Question \ref{q-2} in the remarks after \cite[Question 1.2]{RYY}. Subsequently,  their \cite[Theorem 1.2]{RYY2} gave an explicit expression of the isomorphism between Dolbeault cohomologies in the blow-up formula  to implicitly obtain $(\textbf{$B_n$})\Leftrightarrow(\textrm{$H_{n,2}$})$ via Proposition \ref{blowup0} indeed.
D. Angella, T. Suwa, N. Tardini and A. Tomassini \cite[Theorem 13, Questions 22-24]{ASTT} also studied this equivalence by the \u{C}ech-Dolbeault cohomology with  additional hypotheses and generalized their results to compact complex orbifolds.
In his PhD thesis, by Angella--Tomassini's characterization \cite[Theorems A and B]{AT},  J. Stelzig \cite[Corollary F]{St1} claimed that the $\partial\bar{\partial}$-lemma property is a bimeromorphic
invariant of compact complex manifolds if and only if every submanifold
of a $\partial\bar{\partial}$-manifold is again a $\partial\bar{\partial}$-manifold.
Inspired by them, we will prove the following theorem.
\begin{thm}\label{F}
For any integer $k\in\{1, 2,\cdots,n\}$, there holds the implication hierarchy
$$(\textrm{$B_{n+k}$})\Rightarrow(\textrm{$H_{n+k,k+1}$})\Rightarrow(\textrm{$H_{n}$}).$$
Moreover, $(\textrm{$H_{n,2}$})\Rightarrow (\textbf{$B_n$})$.
\end{thm}

\subsection*{Acknowledgements}
 The author would like to express his sincere gratitude to  Prof. Sheng Rao for explaining the details in their original manuscript on $(\textbf{$B_n$})\Leftrightarrow(\textrm{$H_{n,2}$})$.
 The author is supported by the National Natural Science Foundation of China (Grant No. 12001500, 12071444) and the Natural Science Foundation of Shanxi Province of China (Grant No. 201901D111141).


\section{Preliminaries}
\subsection{A criterion on the $\partial\bar{\partial}$-lemma}
For a compact complex manifold $X$, a natural filtration on the complex $A^\bullet(X)_{\mathbb{C}}$ of $\mathbb{C}$-valued smooth forms on $X$ is defined  as
\begin{displaymath}
F^pA^k(X)_{\mathbb{C}}=\bigoplus\limits_{\substack{r+s=k\\r\geq p}}A^{r,s}(X),
\end{displaymath}
for all $k$, $p$, which  give a spectral sequence $(E_r^{p,q}, F^p H^k(X,\mathbb{C}))$, namely, the \emph{Fr\"{o}licher spectral sequence} of $X$. Then $E_1^{p,q}=H^{p,q}_{\bar{\partial}}(X)$ and
\begin{equation}\label{rep}
F^p H^k(X,\mathbb{C})=\{[\alpha]\in H^k(X,\mathbb{C})|\alpha\in F^pA^k(X)\mbox{ and }d\alpha=0\}.
\end{equation}
Clearly, $F^p H^k(X,\mathbb{C})=0$ for $p<0$ or $p>k$. For convenience, we call $F^\bullet H^k(X,\mathbb{C})$ the \emph{Hodge filtration} on $H^k(X,\mathbb{C})$. Set $V^{p,q}(X)=F^pH^k(X, \mathbb{C})\cap\overline{F}^qH^k(X, \mathbb{C})$ for $p+q=k$, where $\overline{F}^qH^k(X, \mathbb{C})$ is the complex conjugation of the complex subspace $F^qH^k(X, \mathbb{C})$ in $H^k(X, \mathbb{C})$. We say that \emph{the Hodge filtration  gives a Hodge structure of weight $k$ on $H^k(X,\mathbb{C})$}, if
\begin{equation}\label{Hodge decom}
H^k(X, \mathbb{C})=\bigoplus_{p+q=k}V^{p,q}(X),
\end{equation}
and
\begin{equation}\label{sym}
\overline{V^{p,q}(X)}=V^{q,p}(X), \mbox{ for any }p+q=k.
\end{equation}

P. Deligne, Ph. Griffiths, J. Morgan and D. Sullivan established the well-known criterion on the $\partial\bar{\partial}$-lemma as follows.

\begin{thm}[{\cite[(5.21)]{DGMS}}]\label{DGMS}
For a compact complex manifold $X$, the following statements are equivalent:

$(1)$ $X$ satisfies the $\partial\bar{\partial}$-lemma.

$(2)$ $(a)$ The Fr\"{o}licher spectral sequence of $X$ degenerates at $E_1$, and

\quad \mbox{ }\mbox{ }$(b)$ the Hodge filtration  gives a Hodge structure of weight $k$ on $H^k(X,\mathbb{C})$, for every $k\geq 0$.
\end{thm}

\begin{rem}\label{Fro-rem}
For a compact complex manifold $X$, denote  by $b_k(X)$, $h^{p,q}(X)$ the $k$-th Betti, $(p,q)$-th Hodge numbers respectively.

$(1)$ In general, $b_k(X)\leq\sum\limits_{\substack{p+q=k}}h^{p,q}(X)$ for all $k$.

$(2)$ The statement of Theorem \ref{DGMS} $(2)(a)$ is equivalent to that $F^p H^k(X,\mathbb{C})/F^{p+1} H^k(X,\mathbb{C})\cong H^{p,k-p}_{\bar{\partial}}(X)$ for all $k$, $p$, and hence is equivalent to that $b_k(X)=\sum\limits_{\substack{p+q=k}}h^{p,q}(X)$ for all $k$.
\end{rem}

We refer to \cite[Sect. 1.5]{ASTT} and \cite[Sect. 2.3]{RZ} for more discussions on the Fr\"{o}licher spectral sequence and the Hodge structure.

\subsection{Some notations}
Assume that $X$ is a complex manifold with complex dimension $n$.
Denote by $\mathcal{D}^{\prime p,q}(X)$ the space of $(p,q)$-currents  on $X$, which is defined as the dual of the topological vector space $A^{n-q,n-q}(X)$ equipped with its natural topology.
The operators $\partial$ and $\bar{\partial}$ on $A^{\bullet,\bullet}(X)$ naturally induce two differentials $\partial$ and $\bar{\partial}$ on $\mathcal{D}^{\prime \bullet,\bullet}(X)$. 
Evidently, $(A^{\bullet,\bullet}(X), \partial, \bar{\partial})$ and $(\mathcal{D}^{\prime \bullet,\bullet}(X), \partial, \bar{\partial})$ are both double complexes. 
Denote by $H^{q}(\mathcal{D}^{\prime p,\bullet}(X))$ the $q$-th cohomology of the complex $(\mathcal{D}^{\prime p,\bullet}(X), \bar{\partial})$.
The natural inclusion $A^{p,\bullet}(X)\hookrightarrow\mathcal{D}^{\prime p,\bullet}(X)$  induces an isomorphism $\rho_X :H_{\bar{\partial}}^{p,q}(X)\tilde{\rightarrow} H^{q}(\mathcal{D}^{\prime p,\bullet}(X))$.

Let $f:X\rightarrow Y$ be a proper holomorphic map between complex manifolds.
Set $r=\textrm{dim}_{\mathbb{C}}X-\textrm{dim}_{\mathbb{C}}Y$.
The pushforward $f_*:\mathcal{D}^{\prime \bullet,\bullet}(X)\rightarrow \mathcal{D}^{\prime \bullet-r,\bullet-r}(Y)$ of the currents defines a morphism $f_*:H^{q}(\mathcal{D}^{\prime p,\bullet}(X))\rightarrow H^{q-r}(\mathcal{D}^{\prime p-r,\bullet}(Y))$ for any $p$, $q$. 
For convenience, we also denote by $f_*$ the morphism $\rho_Y\circ f_*\circ \rho_X^{-1}:H_{\bar{\partial}}^{p,q}(X)\rightarrow H_{\bar{\partial}}^{p-r,q-r}(Y)$.

\section{The Hodge structures on blow-ups and projective bundles}
\subsection{Blow-up cases}
Let $\pi:\widetilde{X}\rightarrow X$ be the blow-up of  a compact complex manifold $X$ along  a complex submanifold $Y$ and  $E$ the exceptional divisor.
Set $r=\textrm{codim}_{\mathbb{C}}Y\geq 2$ and assume that $i_E:E\rightarrow \widetilde{X}$ is the inclusion.
Let $t\in \mathcal{A}^{1,1}(E)$ be a Chern form of the universal line bundle $\mathcal{O}_{E}(-1)$ on $E={\mathbb{P}(N_{Y/X})}$.
Define a double complex
\begin{displaymath}
K^{\bullet,\bullet}=A^{\bullet,\bullet}(X)\oplus \bigoplus_{i=1}^{r-1}A^{\bullet-i,\bullet-i}(Y).
\end{displaymath}
and a morphism of bounded double complexes
\begin{displaymath}
\psi:K^{\bullet,\bullet}\rightarrow \mathcal{D}^{\prime \bullet,\bullet}(\widetilde{X})
\end{displaymath}
as
\begin{displaymath}
(\alpha, \beta^1,\ldots,\beta^{r-1})\mapsto\pi^*\alpha+\sum_{i=1}^{r-1}i_{E*}\left(t^{i-1}\wedge(\pi|_E)^*\beta^i\right),
\end{displaymath}
where $\alpha\in A^{\bullet,\bullet}(X)$ and $\beta^i\in A^{\bullet-i,\bullet-i}(Y)$.
By \cite[Theorem 1.2]{M}, $\psi$ induces an isomorphism
\begin{equation}\label{blowup1}
H_{\bar{\partial}}^{\bullet,\bullet}(X)\oplus \bigoplus_{i=1}^{r-1}H_{\bar{\partial}}^{\bullet-i,\bullet-i}(Y)\tilde{\rightarrow} H_{\bar{\partial}}^{\bullet,\bullet}(\widetilde{X}),
\end{equation}
i.e., the isomorphism on $E_1$-pages between   the spectral sequences associated to $K^{\bullet,\bullet}$ and $\mathcal{D}^{\prime \bullet,\bullet}(\widetilde{X})$.
Hence  $\psi$     induces an isomorphism $H^k(X,\mathbb{C})\oplus\bigoplus\limits_{i=1}^{r-1}H^{k-2i}(Y,\mathbb{C})\tilde{\rightarrow }H^k(\widetilde{X},\mathbb{C})$
with the isomorphism on the Hodge filtrations
\begin{equation}\label{blowup2}
F^\bullet H^k(X,\mathbb{C})\oplus\bigoplus_{i=1}^{r-1}F^{\bullet-i}H^{k-2i}(Y,\mathbb{C})\tilde{\rightarrow }F^\bullet H^k(\widetilde{X},\mathbb{C})
\end{equation}
for any $k$.
Moreover, $\psi$   induces an isomorphism
\begin{displaymath}
V^{p,q}(X)\oplus\bigoplus_{i=1}^{r-1}V^{p-i,q-i}(Y)\tilde{\rightarrow} V^{p,q}(\widetilde{X})
\end{displaymath}
for any $p$, $q$.

\begin{lem}\label{F-H-bu}
For a given $k$, the Hodge filtration  gives a Hodge structure of weight $k$  on $H^k(\widetilde{X},\mathbb{C})$, if and only if, the Hodge filtrations  give a Hodge structure of weight $k$  on $H^k(X,\mathbb{C})$ and a Hodge structure of weight $k-2i$  on $H^{k-2i}(Y,\mathbb{C})$ for all $1\leq i\leq r-1$.
\end{lem}

By (\ref{blowup1}), (\ref{blowup2}) and Remark \ref{Fro-rem}, we easily obtain
\begin{lem}[{\cite[Theorem 1.6]{RYY}}]\label{E-bu}
The Fr\"{o}licher spectral sequence of $\widetilde{X}$ degenerates at $E_1$, if and only if, so do those of $X$ and $Y$.
\end{lem}

Combining Lemmas  \ref{F-H-bu}, \ref{E-bu} and Theorem \ref{DGMS}, we get
\begin{prop}\label{blowup0}
Let  $\widetilde{X}$ be the blow-up of a  compact complex manifold $X$ along a complex submanifold $Y$ of complex codimension $\geq 2$. Then $\widetilde{X}$ satisfies the  $\partial\bar{\partial}$-lemma, if and only if, $X$ and $Y$ do.
\end{prop}

\begin{rem}
S. Rao, S. Yang, X.-D. Yang \cite[Theorem 1.6]{RYY} \cite[Theorem 1.2]{RYY2} first understood Proposition \ref{blowup0} from the viewpoint of Deligne--Griffiths--Morgan--Sullivan's criterion  for the $\partial\bar{\partial}$-lemma and S. Yang, X.-D. Yang \cite[Theorem 1.3]{YY} studied it from the viewpoint of Angella--Tomassini's characterization for the case of threefolds.
Shortly, D. Angella, T. Suwa, N. Tardini, A. Tomassini \cite[Theorem 13]{ASTT} also considered it by use of the \u{C}ech-Dolbeault cohomology under some additional assumptions.
Eventually, J. Stelzig  obtianed a blow-up formula for Bott-Chern cohomology and wrote this result out explicitly in  \cite[Corollary 1.40]{St1} \cite[Theorems A and B]{AT}.
\end{rem}

\begin{rem}
S. Rao, S. Yang, X.-D. Yang \cite[Theorem 1.2]{RYY2} gave an isomorphism for blow-up in the inverse direction of $\psi$  as
\begin{displaymath}
\phi:H_{\bar{\partial}}^{\bullet,\bullet}(\widetilde{X})\tilde{\rightarrow}H_{\bar{\partial}}^{\bullet,\bullet}(X)\oplus \bigoplus_{i=1}^{r-1}H_{\bar{\partial}}^{\bullet-i,\bullet-i}(Y),
\end{displaymath}
\begin{displaymath}
\alpha\mapsto(\pi_*\alpha, \beta^1,\ldots,\beta^{r-1}),
\end{displaymath}
where  $i_E^*\alpha=\sum\limits_{i=0}^{r-1}h^i\cup(\pi|_E)^*\beta^{i}$ for unique $\beta^{i}\in H_{\bar{\partial}}^{\bullet-i,\bullet-i}(Y)$, $0\leq i\leq r-1$ and $h=[t]_{\bar{\partial}}\in H_{\bar{\partial}}^{1,1}(E)$. 
Actually, $\phi$ can also be lifted to a morphism between complexes of the spaces of forms and currents, see \cite[Lemma 6.5]{M2}. 
Using this morphism, we can also give the relationship between $V^{p,q}(X)$, $V^{p,q}(Y)$ and $V^{p,q}(\tilde{X})$ by above progress.
\end{rem}

As we know, the exceptional divisor for the blow-up $\widetilde{X}$ of $X$ along $Y$ is biholomorphic to the projective bundle of the normal bundle over $Y$ in $X$.
By Proposition \ref{blowup0} and the following Proposition \ref{Proj-bundle-lemma}, we easily get
\begin{cor}
Let  $\widetilde{X}$ be a blow-up of a  complex manifold $X$ along a smooth center with the exceptional divisor $E$. Then $\widetilde{X}$ is a $\partial\bar{\partial}$-manifold, if and only if, $X$ and $E$ are both  $\partial\bar{\partial}$-manifolds.
\end{cor}

\subsection{Projective bundle cases}
Let $\pi:\mathbb{P}(E)\rightarrow X$ be the projective bundle associated to a holomorphic vector bundle $E$ of rank $r$ over a compact complex manifold $X$.
Denote by $t\in \mathcal{A}^{1,1}(\mathbb{P}(E))$ a Chern form of $\mathcal{O}_{\mathbb{P}(E)}(-1)$.
Define a morphism
\begin{displaymath}
\mu=\sum_{i=0}^{r-1}t^i\wedge\pi^*(\bullet) :\bigoplus_{i=0}^{r-1}A^{\bullet-i,\bullet-i}(X)\rightarrow A^{\bullet,\bullet}(\mathbb{P}(E))
\end{displaymath}
of bounded double complexes.
Then $\mu$  induces an isomorphism on $E_1$-pages of the spectral sequences, see \cite[Proposition 3.3]{RYY}, \cite[Proposition 11]{ASTT} or \cite[Corollary 3.2]{M}. With the similar arguments as Sect. 3.1, we can prove following results
\begin{lem}\label{F-H-projbun}
For a given $k$, the Hodge filtration  gives a Hodge structure of weight $k$  on $H^k(\mathbb{P}(E),\mathbb{C})$, if and only if, the Hodge filtration  gives a Hodge structure of weight $k-2i$  on $H^{k-2i}(X,\mathbb{C})$.
\end{lem}

\begin{lem}\label{E-projbun}
The Fr\"{o}licher spectral sequence of $\mathbb{P}(E)$ degenerates at $E_1$, if and only if, so does that of $X$.
\end{lem}

\begin{prop}\label{Proj-bundle-lemma}
Let $\mathbb{P}(E)$ be the projective bundle associated to a holomorphic vector bundle $E$ on a compact complex manifold $X$. Then $\mathbb{P}(E)$ is a $\partial\bar{\partial}$-manifold, if and only if, $X$ is a $\partial\bar{\partial}$-manifold.
\end{prop}

\begin{rem}
The part of `` if " in Proposition \ref{Proj-bundle-lemma} was also proved by D. Angella \emph{et al.} \cite[Corollary 12]{ASTT} in a different way.
\end{rem}

\section{A proof of Theorem \ref{F}}
\begin{proof} Here we just prove $(\textrm{$H_{n+k,k+1}$})\Rightarrow(\textrm{$H_{n}$})$ and the others are the direct corollary of  Proposition \ref{blowup0} and the weak factorization theorem \cite[Theorem 0.3.1]{AKMW}.

Let $X$ be a $\partial\bar{\partial}$-manifold  and $Y$ arbitrary closed complex submanifold of codimension $\geq 1$ in $X$. Note that $X\times \mathbb{C}P^k$ is the projective bundle associated to the trivial bundle $X\times \mathbb{C}^{k+1}$ over $X$ and thus satisfies the $\partial\bar{\partial}$-lemma by Proposition \ref{Proj-bundle-lemma}. Denote by $\{\textrm{pt}\}$ a set consisting of a single point in $\mathbb{C}P^k$. Then $Y\cong Y\times \{\textrm{pt}\}$ has the codimension $\geq k+1$ in $X\times \mathbb{C}P^k$ and  satisfies the $\partial\bar{\partial}$-lemma by $(\textrm{$H_{n+k,k+1}$})$.

%
%
\end{proof}


\end{document}